\documentclass{amsart}
\usepackage[dvips]{color}
\usepackage{graphicx}
\usepackage{epsfig}
\usepackage{tikz}
\usepackage{enumerate}
\usepackage{enumitem}
\usepackage{color}

\usepackage{hyperref}

\usepackage{latexsym}
\usepackage{amssymb}
\usepackage{amsmath}
\usepackage{amsthm}
\usepackage{amscd}
\usepackage{latexsym}
\usepackage{amssymb}
\usepackage{amsmath}
\usepackage{amsthm}
\usepackage{amscd}
\theoremstyle{plain}
\newtheorem{thm}{Theorem}[section]
\newtheorem{lem}[thm]{Lemma}
\newtheorem{prop}[thm]{Proposition}

\theoremstyle{definition}
\newtheorem*{definition*}{Definition}
\newtheorem{defin}[thm]{Definition}
\newtheorem*{notation*}{Notation}

\newcommand{\future}[1]{{}}

\title[The typical measure preserving transformation is not an IET]{The typical measure preserving transformation is not an interval exchange transformation}
\author{Jon Chaika and Diana Davis}

\begin{document}
\maketitle
\vspace{-0.25in}
\centerline{\today}

\section{Introduction and main results}

Two basic problems in mathematics are: when are two objects in a given class the same, and what are the properties of the typical object in a given class? 
A \emph{measure-preserving dynamical system} is a 4-tuple $(X,\mathcal{M},\mu,T)$ where $X$ is a set, $\mathcal{M}$ is a $\sigma$-algebra, $\mu$ is a measure on 
$(X,\mathcal{M})$ and $T:X \to X$ is an $\mathcal{M}$-measurable transformation satisfying $T_*\mu=\mu$. One notion of ``sameness'' in measure preserving systems is isomorphism: 
$(X,\mathcal{M},\mu,T)$ and $(Y,\mathcal{N},\nu,S)$ are \emph{isomorphic} if there exists \mbox{$\phi:X\to Y$}, defined $\mu$-almost everywhere, and $\phi^{-1}:Y \to X$, defined $\nu$-almost everywhere, so that $\phi_*\mu=\nu$ and $S= \phi \circ T \circ \phi^{-1}$.  A basic problem in ergodic theory is to find invariants that distinguish measure preserving systems. 

For the second question, since the 1940s, one way this has been interpreted is by considering the space of measure preserving transformations and calling a property \emph{typical} if it holds on a dense $G_\delta$, or \emph{residual}, set \cite{hal mix, Rokhlin no mix}.

Let $\lambda$ denote Lebesgue measure and let
$$\mathcal{X}=\{T:[0,1] \to [0,1]\ |\ T \text{ is bimeasurable and preserves }\lambda\}/\sim$$
where $\sim$ is the relation
$$(T\sim S \text{ if }Tx=Sx \,\text{ for } \lambda \text{ almost every } x).$$
Consider the topology generated  by 
$$N(T,A,\epsilon)=\{S\in \mathcal{X}: \lambda(SA \Delta TA)<\epsilon\}$$ 
where $A \subset [0,1)$ is measurable  and $\epsilon>0$. 
This called the \emph{weak} topology on $\mathcal{X}$, and turns $\mathcal{X}$ into a Polish space. This topology coincides with the $L^1(\lambda)$ topology restricted to $\mathcal{X}$ (which is a closed subset of $L^1(\lambda)$).

An \emph{interval exchange transformation} (IET) is an invertible  piecewise orientation preserving isometry of the interval, such that the interval is divided into $d$ left closed, right open subintervals that are permuted (see Figure \ref{fig:iet_graph}).  Note that if $T$ is an IET that is continuous on an interval $J$, it is automatically an (orientation preserving) isometry on $J$. 
Our main result is that the typical measure preserving transformation is not isomorphic to any IET:
\begin{thm}\label{thm:main} $\{S\in \mathcal{X}: \text{ there does not exist }T \text{ an IET with }S \text{ isomorphic to }T\}$  contains a dense $G_\delta$ set. 
\end{thm}
We remark that the typical measure preserving system shares many properties with some (often most or all) IETs. The typical measure preserving system is weakly mixing but not strongly mixing, rigid, and not simple. There exist many IETs sharing these properties, so we need different invariants to distinguish the typical measure preserving transformation from all IETs.

\begin{figure}[!h]
\begin{center}
\includegraphics[width=0.5\textwidth]{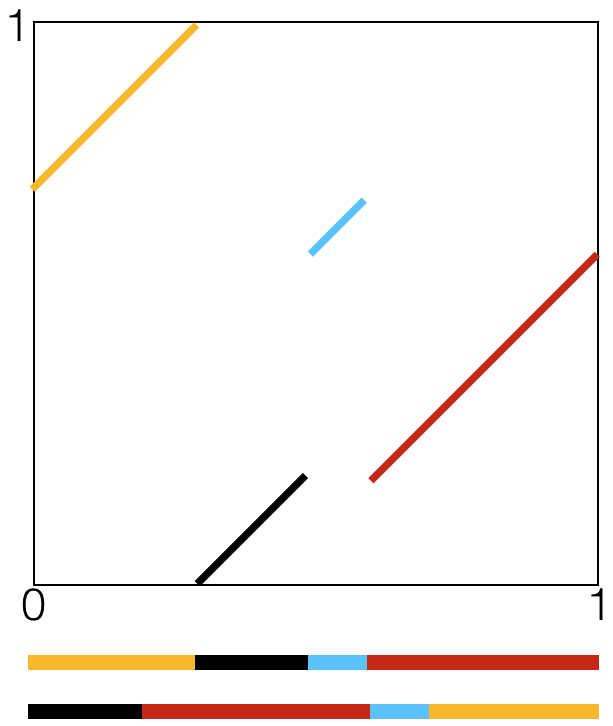}
\caption{An interval exchange transformation, and its graph}
\label{fig:iet_graph}
\end{center}
\end{figure}

\begin{defin}We say $a_1< a_2 < \ldots$ is a \emph{mixing sequence} for $T$ if for any measurable sets $A,B$  we have that 
$$\underset{i \to \infty}{\lim}\, \lambda(T^{-a_i}A\cap B)=\lambda(A)\lambda(B).$$ 
If $T$ has a mixing sequence, it is called \emph{weakly mixing}, and if $\mathbb{N}$ is a mixing sequence for $T$, then $T$ is called \emph{mixing}.
\end{defin}
Halmos proved that a residual subset of $\mathcal{X}$ is weakly mixing \cite{hal mix}, and Rokhlin proved that  a residual set of $\mathcal{X}$ is not mixing \cite{Rokhlin no mix}. 
 It is our understanding that these two results are historically how people learned that there are weakly mixing transformations that are not mixing. (For an explicit example, see \cite{chacon}.)

\begin{defin} Given an increasing function $f:\mathbb{N} \to \mathbb{N}$, we say that an increasing sequence of natural numbers $a_1<a_2< \ldots$ is \emph{$f$-thick} if 
$[j,f(j)]\subset\{a_i\}_{i=1}^\infty$ for infinitely many $j$.
\end{defin}

\begin{prop}\label{prop:thick residual} Let $f$ be any increasing function. Then there exists $G$, a residual subset of $\mathcal{X}$ (with the weak topology), so that every element of $G$ has an $f$-thick mixing sequence. 
\end{prop}
 The next proposition requires a definition: A transformation is \emph{minimal} if the orbit of every point is dense. 
\begin{prop}\label{prop:no-jj-thick}
No minimal IET has a $j \to j^j$-thick mixing sequence. 
\end{prop}

In the language of the introductory paragraphs, certain types of mixing sequences are invariants  that distinguish the typical measure preserving transformation from all interval exchange transformations.

\begin{proof}[Proof of Theorem \ref{thm:main} assuming Propositions \ref{prop:thick residual} and \ref{prop:no-jj-thick}]
Consider $f(j)=j^j$. By Proposition \ref{prop:thick residual}, $ G $,  a residual set of measure preserving transformations has an $f$-thick mixing sequence. Every non-minimal IET has at least two invariant components of positive measure \cite[Theorem 2.16]{bosh}. Letting $A$ and $B$ be two such components, the indicator functions of $A$ and $B$ show that $T$ can not have \emph{any} mixing sequence. 
By Proposition \ref{prop:no-jj-thick}, no minimal IET has an $f$-thick mixing sequence. Clearly isomorphic transformations have the same mixing sequences, and so no measure preserving transformation in $G$  is isomorphic to any IET. 
\end{proof}

\section{Proof of Proposition \ref{prop:thick residual}}
Proving that a property is typical in the space of measure preserving transformations has a standard strategy: First, one shows that it is a $G_\delta$ property. This often comes from studying approximating conditions which are open. Second, one shows that there exists an aperiodic transformation satisfying the property, and quotes the following result of Halmos:
\begin{lem} \cite[Theorem 1]{hal mix} The conjugacy class of any aperiodic transformation in $\mathcal{X}$ is dense in $\mathcal{X}$. 
\end{lem}

We now follow the strategy outlined above.  

\begin{lem}\label{lem:gdelta} For any increasing function $f$, the subset of $\mathcal{X}$ that has an $f$-thick mixing sequence is a $G_\delta$ set.
\end{lem}
The proof uses some straightforward facts that we leave as an exercise:
\begin{enumerate} 
\item For $A,B$ measurable and $n \in \mathbb{Z}$, the function defined by $\Phi_{A,B;n}:\mathcal{X} \to \mathbb{C}$ by $\Phi_{A,B;n}(T)=\lambda(T^{-n}A \cap B)$ is continuous. 
\item For any $j,k \in \mathbb{N}$ and $\epsilon>0$, the set 
$$U_{A,B;j,k}(\epsilon):=\{T \in \mathcal{X}:|\lambda(T^{-n}A \cap B)-\lambda(A)\lambda(B)|<\epsilon\, \text{ for all }j\leq n\leq k\}$$ is open. 
\item $\lambda(T^{-n_i}A\cap B)-\lambda(A)\lambda(B) \to 0$ for all measurable sets $A,B$ iff $\lambda(T^{-n_i}I\cap J)-\lambda(I)\lambda(J)$ for all dyadic intervals $I,J$. 
\end{enumerate}
\begin{proof}[Proof of Lemma \ref{lem:gdelta}]
Let $D_1,...$ be an enumeration of the (countable) set of dyadic intervals. The subset of elements of $\mathcal{X}$ that have $f$-thick  mixing sequences is
$$ \bigcap_{m=1}^\infty \, \bigcap_{n=1}^\infty \,  \bigcap_{b=1}^{\infty}  \, \bigcup^{\infty}_{j=b}  \bigcap_{k,\ell<m}U_{D_k,D_\ell; j,f(j)}\left(\frac 1 n\right).$$
This is by construction a countable intersection of open sets and therefore is a $G_\delta$ set. 
\end{proof}

\begin{proof}[Proof of Proposition \ref{prop:thick residual}] By the previous two lemmas it suffices to show that there exists an aperiodic measure preserving transformation 
with an $f$-thick mixing sequence. As there are mixing transformations, (and $\mathbb{N}$ is $f$-thick for any $f$) there exists $T \in \mathcal{X}$ so that $T$ has an $f$-thick mixing sequence for all increasing $f$. 
\end{proof}

\section{Proof of Proposition \ref{prop:no-jj-thick}}

To prove Proposition \ref{prop:no-jj-thick}, which requires showing that certain sequences can not be mixing sequences for IETs, we take advantage of a behavior at the opposite extreme of mixing sequences: rigidity sequences. We say $n_1,...$ is a \emph{rigidity sequence for $T$} if $T^{n_i}$ converges (in $L^1$) to the identity. Surprisingly, the typical measure preserving transformation not only has a mixing sequence, but it also has a rigidity sequence, which necessarily intersects its mixing sequence in at most a finite set. We will use a mild generalization of rigidity sequences, \emph{partial rigidity sequences}, to prove Proposition \ref{prop:no-jj-thick}. (Some IETs do not have rigidity sequences, let alone large enough rigidity sequences to rule out $f$-thick mixing sequences.) 


We say that $n_1,...$ is a \emph{$c$-partial rigidity sequence} if there exist measurable sets $A_1,...$ so that $\lambda(A_i)>c$ for all $i$ and 
$\underset{j \to \infty}{\lim}\,  \underset{x\in A_j}{\sup}\, d(T^{n_j}x,x)=0$.

\begin{lem} To prove Proposition \ref{prop:no-jj-thick}, it suffices to prove that for each minimal IET, $T$, there exists an integer $r>1$, a real number $c>0$, and a sequence $n_1,....,$ so that $n_{i+1}<r\,n_i$ for all $i$, and $n_1<...$ is a $c$-partial rigidity sequence for $T$. 
\end{lem}
\begin{proof}

To prove the lemma, it suffices to show that if $n_1,...$ is a $c$-partial rigidity sequence for $T$, then it cannot be contained in a mixing sequence. We now prove this. Let $2^{-k}<\frac 1 4 c$ and $\epsilon<\frac 1 {2^{k+4}}$. By the pigeonhole principle, for each $j$ there exists $b\in \{0,...,2^k\}$ (which depends on $j$) so that $\lambda\left(A_j \cap \left[\frac b {2^k},\frac {b+1}{2^k}\right]\right) \geq c 2^{-k}$. By the definition of $c$-partial rigidity, there exists a $j_0$ so that $d(T^{n_j}x,x)<\epsilon$ for all $x\in A_j$ and $j\geq j_0$. Now if $j>j_0$ and $x\in \left[\frac b {2^k},\frac {b+1}{2^k}\right]\cap A_j$ then $T^{n_j}x\in \left[\frac b {2^k}-\frac 1 {2^{k+4}},\frac {b+1}{2^k}+\frac 1 {2^{k+4}}\right]$. So 
\begin{multline*}\lambda\left(T^{n_j}\left( \left[\frac b {2^k}-\frac 1 {2^{k+4}},\frac {b+1}{2^k}+\frac 1 {2^{k+4}}\right]\right)\cap  \left[\frac b {2^k}-\frac 1 {2^{k+4}},\frac {b+1}{2^k}+\frac 1 {2^{k+4}}\right]\right)\geq \\
\lambda\left(A_j \cap  \left[\frac b {2^k},\frac {b+1}{2^k}\right]\right)>
\frac c {2^k}>2 \cdot \left(\frac 9 8 \cdot 2^{-k}\right)^2=2\lambda\left( \left[\frac b {2^k}-\frac 1 {2^{k+4}},\frac {b+1}{2^k}+\frac 1 {2^{k+4}}\right]\right)^2.
\end{multline*} 
There exists $b$ so that this occurs infinitely often, and so $n_1,...$ cannot be a mixing sequence. 
\end{proof}

 The next proposition, whose proof takes the remainder of the paper, completes the proof of Proposition \ref{prop:no-jj-thick}, and thus of Theorem \ref{thm:main}.

Notation: Let $|J|$ denote the length of an interval, and let $\lambda(J)$ denote the Lebesgue measure of a set. 
\begin{prop}\label{prop:keane-rigidity} Let $T$ be a $d$-IET  that is minimal. For any $\epsilon>0$ there exists $n_0$ so that for all $n\geq n_0$ there exists 
\begin{itemize}
\item $k \in [\frac{n}{20 \cdot d},20nd]$
\item $A \subset [0,1)$ with $\lambda(A)>\frac 1 {10^5d^5}$
\end{itemize}
so that for all $x\in A$, $d(T^kx,x)<\epsilon.$
\end{prop}

The plan for the rest of the paper is as follows: Our proof of Lemma \ref{lem:easy rig} establishes Proposition \ref{prop:keane-rigidity} under mild assumptions. It is reminiscent of \cite{kat no mix} and \cite{veech metric}. Lemmas \ref{lem:sizable} and \ref{lem:possibility_2} establish Proposition \ref{prop:keane-rigidity} when the mild assumptions do not hold, and show that in this case one can build a set $A$ by other means. 

Let $D$ be the set of discontinuities of $T$. Choose $n_0$ so that 
\begin{equation}\label{eq:n size} \bigcup_{i=0}^{\lfloor \frac{n_0}2\rfloor}T^{-i}D \text{ is $\epsilon$ dense.}
\end{equation}
Note that $n_0$ exists because by the assumption that $T$ is minimal, $\{T^{-i}x\}_{i\in \mathbb{N}}$ is dense for all $x$, and in particular if $\delta\in D$.

\begin{lem}\label{lem:easy rig} If there exists an interval $J$ so that 
\begin{enumerate}[label=(\alph*)]
\item\label{cond:cont} $T^{i}$ is continuous (and therefore an isometry) on $J$ for all $0\leq i\leq m$, and
\item \label{cond:disjoint} $T^iJ \cap J=\emptyset$ for all $0\leq i\leq m-2 $,
\end{enumerate}
then there exists $A$ with $\lambda(A)>|J|\frac {m-2} {2(d+2)}$ and $0< j\leq \frac{2}{ |J|}$ so that for all $x\in A$,
$$d(T^jx,x)<|J|.$$
\end{lem}
To prove this result we use the \emph{Poincar\'e first return map}, a standard construction in ergodic theory. 
Let $(X,\mathcal{M},\nu,S)$ be a probability measure preserving dynamical system and let $A \in \mathcal{M}$. For $x\in A$ the \emph{first return time} of $x$ to $A$ is 
$n_A(x)=\min\{j>0:S^jx \in A\}$  and  by the Poincar\'e recurrence theorem, \cite[Theorem 1, \S 1.1]{CFS} this is defined for $\nu|_A$ almost every $x$. The \emph{first return map} of $S$ to $A$ is $x \to S^{n_A(x)}x$  is measurable and $\nu|_A$ measure preserving \cite[\S 1.5]{CFS}.
In the proof of Lemma \ref{lem:easy rig} we use the following standard result about IETs: \\

\begin{lem}\label{lem:CFS}\cite[Lemma 2, \S 5.3]{CFS} If $T$ is an IET on $d$ intervals, and $I$ is an interval, then the first return map of $T$ to $I$ is an IET on at most $d+2$ intervals. Moreover, on each interval the first return time is constant. 
\end{lem}
The last claim in the lemma is established in the last two sentences in the proof of \cite[Lemma 2, \S 5.3]{CFS}. 
We sketch the idea of the proof of the lemma. Consider how $T^j$ can become discontinuous on a subinterval of $I$. It must map $I$ to a discontinuity of $T$. Because we are examining a first return map, the injectivity of $T$ implies that each discontinuity can only cut $I$ once before first return. The endpoints of $I$ can also cause discontinuities of the return time function, which can create two more cuts. 
\begin{proof}[Proof of Lemma \ref{lem:easy rig}]
Consider the first return map of $T$ to $J$. 
It consists of at most $d+2 $ intervals, and on each such interval, the return time is constant. 
Consider the set $S$ of points whose return time is at most $\frac 2 {|J|}$. 
$S$ has measure at least $\frac{|J|}2$,  because otherwise $\lambda\left(\bigcup_{i=0}^{\frac 2 {|J|}}T^i\left(\left\{x:n_J(x)>\frac 2 {|J|}\right\}\right)\right)\geq \left(\frac 2 {|J|}+1\right)\frac {|J|}2>1$. 
$S$ is divided into at most $d+2$ intervals, so one of them has measure at least $\frac{|J|}2\frac 1 {d+2}$.  Let $J'$ be such an interval.

We claim that we can choose $A$ to be $\cup_{i=0}^{m-2}T^i(J')$, and $j$ to be the return time of $J'$ to $J$. 
Indeed if $x \in J'$, then 
\begin{equation}\label{eq:return}T^jx\in T^jJ'\subset J \text{ and so }d(T^jx,x)\leq |J|.
\end{equation} 

Also if $x\in T^i J'$ for some  $0\leq i\leq m$, then $T^{-i}x\in J'$, and \eqref{eq:return} implies \mbox{$T^jT^{-i}x\in J$}, and so 
$$ d(T^jT^{-i}x,T^{-i}x)\leq|J|$$
$T^i$ is an isometry on $J$ (by \ref{cond:cont}), giving
\begin{equation}\label{eq:up tower} d(T^jx,x)\leq |J|
\end{equation}
 Because $J,...,T^{m-2}J$ are disjoint, we have that $J',...,T^{m-2}J'$ are too, establishing  $\lambda(A)\geq (m-1)\frac{|J|}2\frac 1 {d+2}$ and the lemma. 
\end{proof}

The idea of  \eqref{eq:up tower}, that is, using that powers of $T$ are an isometry on $J$ to establish \eqref{eq:return} for images of $J$, will be used frequently below.

\begin{figure}[!h]
\begin{center}
\includegraphics[width=0.8\textwidth]{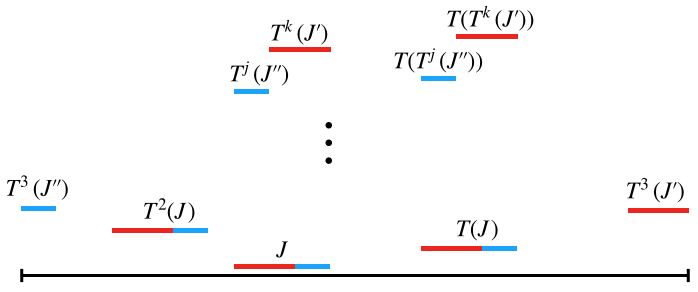}
\caption{The construction in the proof of Lemma \ref{lem:easy rig}}
\label{fig:disjoint_iterates}
\end{center}
\end{figure}

We will see that we can arrange $J$ and $P_S$ satisfying $|J|>\frac 1 {10d^2n}$  and Lemma \ref{lem:easy rig} \ref{cond:cont} with  $m=n$. We now show that even if such an interval $J$ doesn't satisfy Lemma \ref{lem:easy rig} \ref{cond:disjoint},  we can still prove Proposition  \ref{prop:keane-rigidity}. 

Recall that  $ D=\delta_1,...,\delta_{d-1}$ is the set of discontinuities of $T$. Let $S=\cup_{i=0}^n T^{-i}D$. 
Let $\mathcal{P}_S$ be the partition of $[0,1)$ into left closed, right open half intervals whose endpoints are in $S\cup\{0,1\}$. 
Observe that any element of $\mathcal{P}_S$ satisfies Lemma \ref{lem:easy rig} \ref{cond:cont}. Given an ordered pair $(\delta,\delta')\in D$ let $\mathcal{P}_n(\delta,\delta')$ be the set of partition elements with left endpoint $a$ and right endpoint $b$ so that $T^{i}a=\delta$ and $\underset{x \to b^{-}}{\lim}\, T^jx=\delta'$ for some $0\leq i,j\leq n$. Note that because our intervals of continuity are right open, we require the limit as $x$ goes to $b$ from the left in place of $b$. 

\begin{definition*}
 We say that an ordered pair 
$(\delta, \delta')$ is \emph{n-sizable} if $$\lambda(\cup_{J \in \mathcal{P}_n(\delta,\delta')}J)\geq \frac 1 {(d-1)^2}.$$
\end{definition*}
By the pigeonhole principle, for every $n$ there is an $n$-sizable pair. 


\begin{lem}\label{lem:sizable} If $(\delta,\delta')$ is $n$-sizable, then one of the following holds:
\begin{enumerate}
\item\label{possibility1} There exists $J$ as in Lemma \ref{lem:easy rig} with $m=\lfloor \frac n 2 \rfloor$ so that $|J|>\frac 1 {10}\frac 1 {d^2}\frac 1 n$, or
\item \label{possibility2} there exists $0< k\leq  \frac n 2$ so that one of the following holds: 

$0< T^k\delta-\delta<\frac 1 {10}\frac 1{d^2n}$ \qquad or \qquad $0<\delta'-\underset{x \to \delta'^+}{\lim} T^kx<\frac 1 {10}\frac 1{d^2n}$. 
\end{enumerate}
\end{lem}

\begin{proof}
Consider the following condition:
\begin{enumerate}
\item[\emph{(i')}] \label{possibility3} There exists $J'=[a,b) \in \mathcal{P}_n(\delta,\delta')$ with  $|J'|> \frac 1 {10}\frac 1{d^2n}$ and so that  $$T^{i}\left[a,a+\frac 1 {10}\frac 1{d^2n}\right) \cap \left[a,a+\frac 1 {10}\frac 1 {d^2n}\right)= \emptyset$$ for all $0<i<\frac n 2 $.
\end{enumerate}

Note that a $J'\in \mathcal{P}_n(\delta,\delta')$ exists satisfying $|J'|> \frac 1 {10}\frac 1{d^2n}$ by our assumption that $(\delta,\delta')$ is $n$-sizable and contains at most $n$ intervals. Clearly (i') implies (i). We will show that if (i') fails, then \eqref{possibility2} holds, so that combining these gives us that if \eqref{possibility1} fails, then \eqref{possibility2} holds.  

Let $J=[a,a+\frac 1 {10}\frac 1{d^2n})$ and assume that $T^kJ \cap J\neq \emptyset$ for some $0<k<\frac n 2 $. It follows that $|T^kx-x| <\frac 1 {10}\frac 1{d^2n}$ for all $x \in J$. For concreteness let's assume \mbox{$0<T^kx-x<\frac 1 {10}\frac 1{d^2n}$}. Now $T^ia=\delta$ for some $i\leq n$. As $T^ka\in J$ (by our concreteness assumption) and $T^i$ is an isometry on $J$, $T^kT^ia=T^iT^ka=T^ia+(T^ka-a)$. As $T^ia=\delta$, we have \eqref{possibility2}.

The case of $0>T^kx-x>-\frac 1 {10}\frac 1{d^2n}$ is the same with $T^k\delta-\delta$ replaced by $\delta'-\underset{z \to \delta'^{-}}{\lim}\, T^kz$. 
\end{proof}

Note that the condition $k\leq n/2$ is used later; see \eqref{eq:use_n2}.

So we have that absence of \eqref{possibility2} implies the existence of an interval $J$ satisfying \ref{cond:disjoint} and $|J|>\frac 1 {10}\frac 1{d^2n}$. As $J\subset J'\in \mathcal{P}_S$ we have \ref{cond:cont}, establishing the assumptions of Lemma \ref{lem:easy rig}.

\begin{lem}\label{lem:possibility_2} 
Let $\epsilon$ be as in Proposition \ref{prop:keane-rigidity}.
Under possibility (ii) of Lemma~\ref{lem:sizable} and assuming $$\max_{I \in \mathcal{P}_n(\delta,\delta')} \, |I|<\epsilon,$$  there exist $k$, $A$ as in Proposition~\ref{prop:keane-rigidity}. 
\end{lem}

\begin{proof}[Proof of Proposition \ref{prop:keane-rigidity} assuming Lemma \ref{lem:possibility_2}]  Lemma \ref{lem:possibility_2} establishes Proposition \ref{prop:keane-rigidity} when Lemma \ref{lem:sizable}  \eqref{possibility2} holds, and so we now prove Proposition \ref{prop:keane-rigidity} in the case of Lemma \ref{lem:sizable} \eqref{possibility1}. 
There exists an interval $J\in \mathcal{P}_S$ with $|J|>\frac 1 {10}\frac 1 {d^2}\frac 1 n$ 
satisfying the assumptions for Lemma \ref{lem:easy rig} with $m=\lfloor \frac n 2 \rfloor$.   
 So there exists $A$ with  $\lambda(A)>|J|\frac {m-2} {2(d+2)}\geq \frac {n/2-2}{2(d+2)}\frac 1 {10}\frac 1 {d^2}\frac 1 n $ so that for all $x\in A$,
$$d(T^jx,x)<|J|.$$ Now $J\in \mathcal{P}_S$ and thus \eqref{eq:n size} implies that its diameter is less than $\epsilon$. Lastly, by Assumption \ref{cond:cont} $j>m\geq \lceil \frac n 2 \rceil$ and by the proof of Lemma \ref{lem:easy rig}  we have $j<2m$. completing the proof. 
\end{proof}

\begin{proof}[Proof of Lemma \ref{lem:possibility_2}] For concreteness we assume that there exists $k\leq n$ so that $$0<T^k\delta-\delta < \frac 1 4\frac 1{10dn}.$$ (The other case is similar.) \\

\noindent
\textbf{Sublemma:} If $J=[a,b)\in \mathcal{P}_n(\delta,\delta')$ and $T^kJ \cap J \neq \emptyset$ then $T^{ik}J=i(T^k\delta-\delta)+J$  for all $ik<n$. 

Notation: if $B\subset \mathbb{R}$ and $c \in \mathbb{R}$ then $c+B=\{c+b:b\in B\}$.

\begin{proof} $T^ia=\delta$ for some $0\leq i\leq n$ by construction. Now (by the case we are considering) $T^ka\in J$  and so $T^iT^ka=T^ka+(T^ia-a)$ because $T^i$ acts as an isometry on $J$. So $T^k\delta=T^kT^ia=\delta+ T^ka-a$. 

We wish to show by induction that for all $i$ satisfying $(i+1)k<n$, we have 
$T^{(i+1)k}J=T^k\delta-\delta+T^{ik}J,$
 which will establish the sublemma. To do this we inductively assume that 
 \begin{equation}\label{eq:assump 1}T^{ik}J=i(T^k\delta-\delta)+J
 \end{equation} and 
 \begin{equation}\label{eq:assump 2} T^{ik}y-T^{(i-1)k}y=T^k\delta-\delta
 \end{equation}
for all $y \in J$. The case of $i=1$ is above.

We now prove the sublemma by induction. Assume it is true for $i=\ell$ and that $(\ell+1)k<n$.  
By our assumption, $T^{k}J \cap J \neq \emptyset$, and so $T^{(\ell+1) k}J\cap T^{\ell k} J\neq \emptyset$. 
So, there exists $y \in T^{(\ell +1)k}J\cap T^{\ell k}J$. Since $y \in T^{\ell k}J$, there exists $x \in J$ so that $T^{\ell k}x=y$, and thus 
\begin{equation}\label{eq:induct1}T^{\ell k}x-x=\ell(T^k\delta-\delta)
\end{equation} by our induction hypothesis \eqref{eq:assump 1}. 
Now we have
$$T^{-k}y \in T^{-k}(T^{(\ell+1)k}J\cap T^{\ell k}J)=T^{\ell k}J\cap T^{(\ell-1)k}J\neq \emptyset$$ and so $T^{k}z-z=T^kx'-x'$ for all $z \in T^{(\ell-1)k}J$ and $x' \in T^{\ell k}J$. By applying \eqref{eq:assump 2} to $z$ we see that this is $T^k \delta-\delta$. This establishes that $T^{(\ell+1)k}x-T^{(\ell+1-1)k}x=T^{k}\delta-\delta$ for all $x\in J$, inductive claim \eqref{eq:assump 2} for $i=\ell+1$. 
Combining this with \eqref{eq:induct1} we see 
that if $x \in J$ then $$T^{(\ell+1)k}x-x=T^{(\ell+1) k} x-T^{\ell k}x+(T^{\ell k}x-x)=T^k\delta-\delta+\ell(T^k\delta-\delta)$$ establishing inductive claim \eqref{eq:assump 1} (for $i=\ell+1$) and the sublemma.
\end{proof}
We now complete the proof of the lemma 
using the sublemma. Let $\mathcal{B}=\{J \in \mathcal{P}_n(\delta,\delta'): T^kJ \cap J \neq \emptyset\}$, the set of intervals to which we can apply the sublemma. Let $\tilde{J}=\cup_{i=0}^{\lfloor \frac n {2k}\rfloor}T^{ik}J$ and $A=\cup_{J\in \mathcal{B}}\tilde{J}$. 

Recall that $D$ is the set of discontinuities of $T$.
No $\tilde{J}$ intersects $T^{-j}D$ for \mbox{$0\leq j\leq \frac n 2$}, because otherwise $T^{-j-\ell k}D\ \cap\ J\neq \emptyset$ for $0\leq \ell \leq \lfloor \frac n {2k}\rfloor$. As \mbox{$0\leq \ell k+j\leq n$}, this contradicts the fact that $J \in \mathcal{P}_S$. So we may assume that these diameters are less than $\epsilon$ (if $n$ is large enough). 
Now by the sublemma (and our concreteness assumption) we have that 
\begin{equation} \label{eq:use_n2}
T^{k \left\lfloor \frac n {2k}\right\rfloor }x-x=\left\lfloor \frac n {2k}\right\rfloor (T^k\delta-\delta)<diam(J)+\left\lfloor \frac n {2k}\right\rfloor(T^k\delta-\delta)=diam(\tilde{J})
\end{equation}
for all $J \in \mathcal{B}$.  Because $\frac{n}{20 \cdot d}\lceil\frac n {2k}\rceil\leq 20nd$, the only thing left to show is that $A$ has the specified measure. This follows because, by our assumption that $(\delta,\delta')$ are $n$-sizable, and the fact that $\#\mathcal{P}_n(\delta,\delta')\leq n+1$, we have that  a set of measure at least $\frac 3 4  \frac 1 {d^2}$ must be contained in intervals, each of whose length is at least $\frac 1 4 \frac 1 {d^2} \frac 1 {n+1}$. Because we are assuming possibility (ii) of Lemma~\ref{lem:sizable}, any such interval is longer than $T^k\delta-\delta$ and thus is in $\mathcal{B}$. 
\end{proof}

\textbf{Acknowledgments:} The research of J. Chaika was supported in part by NSF grants DMS-135500 and DMS-1452762, the Sloan foundation, a Poincar\'e chair, and a Warnock chair. He would also like to thank M. Boshernitzan for asking him this question.

\end{document}